\documentclass[12pt,a4paper,reqno]{amsart}
\usepackage[utf8]{inputenc}
\usepackage[T1]{fontenc}
\usepackage{amsmath}
\usepackage{amsthm}
\usepackage{amssymb}
\usepackage[abbrev]{amsrefs}
\usepackage{mathrsfs}
\usepackage[dvipsnames]{xcolor}
\usepackage{bm}
\usepackage{bbm}
\usepackage{hyperref}
\usepackage[]{enumitem}
\newlist{FNenumerate}{enumerate}{1}
\setlist*[FNenumerate]{label=\textbf{Action\ \arabic*}, resume=FN}
\AtBeginDocument{\def\MR#1{}}
\makeatletter
\@namedef{subjclassname@2020}{%
\textup{2020} Mathematics Subject Classification}
\makeatother
\newtheorem{thm}{}[section]
\newtheorem{theorem}[thm]{Theorem}

\newtheorem{lemma}[thm]{Lemma}

\theoremstyle{definition}
\newtheorem{definition}[thm]{Definition}
\newtheorem{example}[thm]{Example}
\numberwithin{equation}{section}
\allowdisplaybreaks
\newcommand{\Nnorm}[1]{{\left\vert\kern-0.25ex\left\vert\kern-0.25ex\left\vert #1\right\vert\kern-0.25ex\right\vert\kern-0.25ex\right\vert}}
\newcommand{\abs}[1]{\left\lvert#1\right\rvert}
\newcommand{\norm}[1]{\left\lVert#1\right\rVert}
\newcommand{\enbrace}[1]{\left\lbrace#1\right\rbrace}
\newcommand{\enpar}[1]{\left(#1\right)}

\newcommand{\Fou}{\ensuremath{\mathcal{F}}}

\newcommand{\Gt}{\ensuremath{\mathcal{G}}}

\newcommand{\Bt}{\ensuremath{\mathcal{B}}}
\newcommand{\Ct}{\ensuremath{\mathcal{C}}}
\newcommand{\Ft}{\ensuremath{\mathcal{F}}}
\newcommand{\FF}{\ensuremath{\mathbb{F}}}

\newcommand{\XX}{\ensuremath{\mathbb{X}}}
\newcommand{\YY}{\ensuremath{\mathbb{Y}}}

\newcommand{\NN}{\ensuremath{\mathbb{N}}}
\newcommand{\DD}{\ensuremath{\mathbb{D}}}

\newcommand{\ZZ}{\ensuremath{\mathbb{Z}}}

\newcommand{\ft}{\ensuremath{\bm{F}}}
\newcommand{\kt}{\ensuremath{\bm{k}}}
\newcommand{\rt}{\ensuremath{\bm{r}}}
\newcommand{\ff}{\ensuremath{\bm{f}}}

\newcommand{\pp}{\ensuremath{\bm{p}}}
\newcommand{\Nt}{\ensuremath{\mathcal{N}}}
\newcommand{\Vt}{\ensuremath{\mathcal{V}}}
\DeclareMathOperator*{\supess}{sup\, es}
\newcommand{\tq}{\colon}

\begin{document}
\title[Characterization of complete topological vector spaces]{A characterization of complete topological vector spaces with applications to spaces of measurable functions}
\author[Ansorena]{Jos\'e L. Ansorena}
\address{Department of Mathematics and Computer Sciences\\
Universidad de La Rioja\\
Logro\~no\\
26004 Spain}
\email{joseluis.ansorena@unirioja.es}
\author[Marcos]{Alejandro Marcos}
\address{Department of Mathematics and Computer Sciences\\
Universidad de La Rioja\\
Logro\~no\\
26004 Spain}
\email{alejandro.marcos@unirioja.es}
\begin{abstract}
The aim of this paper is twofold. Firstly, we give easy-to-handle criteria to determine whether a given family of subsets of a vector space is neighbourhood basis of the origin for a complete vector topology. Then, we apply these criteria to construct quite general complete topological vector spaces of measurable functions. 
\end{abstract}
\thanks{Alejandro Marcos is supported by an Introductory research grant at the University of la Rioja. Jos\'e L. acknowledges the support of the Spanish Ministry for Science and Innovation under Grant PID2022-138342NB-I00 for \emph{Functional Analysis Techniques in Approximation Theory and Applications}.}
\maketitle
\section{Introduction}\noindent
Complete topological vector spaces naturally arise when studying function spaces. Combining notions from linear algebra and topology, these spaces allow us to apply powerful results from different branches of mathematics into a unified framework. For instance, Banach and Fr\'echet spaces are complete topological vector spaces. As relevant as it is, this concept is also fairly elementary in our theory, and the techniques for proving that a given topology is linear and checking its completeness are well-known. So, most experts do not linger on the basic details of checking if a space verifies the necessary conditions to be a complete topological vector space and delve immediately into more complex results. This lapse in the theory may hinder the reader's understanding of the topic at hand. In this pedagogical note, we propose easy-to-handle tools to prove that a given family of sets on a vector space is a neighbourhood basis at the origin for a complete vector topology. We address this task in Section~\ref{sect:comp}. In Section~\ref{sect:examples}, we give some examples exhibiting the applicability of our tools. Specifically, we construct certain relatives of the space of measurable functions on a measure space, and general quasi-Banach-valued Musielak-Orlicz spaces without assuming either local boundedness or local convexity.

Through this article, we use standard techniques and concepts from functional analysis. We refer the reader to \cite{Rudin1991} for the basics of topological vector spaces and function spaces. To fix the terminology, we point out that we will deal with vector spaces over the real or complex field $\FF$ and assume vector topologies to be Hausdorff.

\section{Complete metric vector spaces}\label{sect:comp}\noindent
A topology on a set can be defined through the neighborhoods of its points. This basic result of general topology makes it important to characterize the families of sets that are neighbourhood bases at the points of a set $\XX$ for a topology on $\XX$.  A neat characterization of such families can be found, for instance, in \cite{Willard1970}. As vector topologies are translational invariance, an analogous result in this setting should be stated in terms of neighbourhoods of the origin. Such a characterization of the families of the sets that are neighbourhood bases at the origin for a vector topology is known (see \cite{NariciBeck2011}). Notwithstanding, we record it for completeness.

\begin{theorem}\label{th:localbasis}
Let $\XX$ be a vector space and $\Bt$ a family of subsets of $\XX$. Then, $\Bt$ is a local basis at the origin of some vector topology on $\XX$ if and only if the following properties hold.
\begin{enumerate}[label=(E.\arabic*)]
\item \label{th:localbasis:1}
Given $U, V\in\Bt$, there exists $B\in\Bt$
such that $B\subseteq U\cap V$.
\item \label{th:localbasis:2}
Given $U\in\Bt$, there exists $V\in\Bt$ such that $V+V\subseteq U$.
\item \label{th:localbasis:3}
Given $U\in\Bt$, there exists $V\in\Bt$ such that
$\DD\cdot V \subseteq U$, where $\DD$ denotes the unit disk on the scalar field.
\item \label{th:localbasis:4}
For each nonzero element $x\in\XX$ there exists $U\in\Bt$ such that $x\notin U$.
\item \label{th:localbasis:5}
Given $x\in\XX$ and $U\in\Bt$, there exists $\varepsilon >0$ such that $\lambda x\in U$ whenever $\abs{\lambda}<\varepsilon$.
\end{enumerate}
\end{theorem}

\begin{proof}
Checking a neighbourhood basis of the origin of a topological vector space to satisfy these properties is routine. We shall focus on proving their sufficiency.

We start by showing the existence of a topology on $\XX$ with
\[
\Fou:= (\Bt_x)_{x\in\XX}, \quad \Bt_x=\enbrace{x+U \tq U\in\Bt}
\]
as a family of local bases. Taking into account \cite{Willard1970}*{Theorem 4.5}, we have to check that the following properties are satisfied.
\begin{enumerate}[label=(\roman*)]
\item\label{Willard:hyp:1}
The vector $x$ lies in $V$ for all $V\in\Bt_x$.
\item\label{Willard:hyp:2}
Given $V_1, V_2\in\Bt_x$, then there is $V_3\in\Bt_x$ such that $V_3\subseteq V_1\cap V_2$.
\item\label{Willard:hyp:3}
Given $V\in\Bt_x$, we can find a set $V_0\in\Bt_x$ such that for all $y\in V_0$ there is $W\in\Bt_y$ such that $W\subseteq V$.
\end{enumerate}
By \ref{th:localbasis:5}, $0\in U$ for all $U\in\Bt$ and thus $x\in x+U$ for all for all $x\in\XX$ and $U\in\Bt$ so that \ref{Willard:hyp:1} holds. Additionally, \ref{Willard:hyp:2} follows from the fact that property \ref{th:localbasis:1} is preserved by translations.

To prove \ref{Willard:hyp:3}, pick $V\in\Bt_x$. By definition, there exists $U\in\Bt$ such that $V=x+U$. Therefore, by \ref{th:localbasis:2}, there is $U_0\in\Bt$ such that $U_0+U_0\subseteq U$ holds. Let us now consider the set $V_0=x+U_0$. For any vector $y\in V_0$ we have $y-x\in U_0$, whence $y-x+U_0\subseteq U$. Consequently,
$W:=y+U_0\subseteq V$. Since $W\in\Bt_y$, we are done.

Once we know the existence of a topology for which $\Fou$ is a family of local bases, we show that this topology is a vector topology. To that end, we have to prove that
\begin{enumerate}[label=(\roman*),resume]
\item \label{hyp:vectortopo:1}
The natural operations of the vector space, namely the sum and the product by scalars, are continuous, and
\item \label{hyp:vectortopo:2}
The set $\enbrace{0}$ is a closed set.
\end{enumerate}

We start by proving \ref{hyp:vectortopo:2}. By \ref{th:localbasis:4}, for every nonzero vector $x$ of $\XX$ we can find a set $U\in\Bt$ verifying $-x\notin U$. Thus, $0=x-x\notin x+U$ and therefore $x+U\subseteq \XX\setminus\{0\}$.

Finally, we address the proof of \ref{hyp:vectortopo:1}. Property~\ref{th:localbasis:2} gives the continuity of the sum at $(0,0)$. We infer from the shape of the local bases $\Bt_x$, $x\in\XX$, that the sum operator is continuous at any point $(x,y)\in\XX^2$. Proving that the product operator is continuous requires some additional effort. Fix $\alpha\in\FF$ and $x\in\XX$. It suffices to show that for all $U\in\Bt$ there is a positive scalar $R$ and a set $V\in\Bt$ such that $D(\alpha, R)\cdot(x+V)\subseteq \alpha x +U$.

Let $N$ be a natural number verifying $\abs{\alpha}<N$ and let $V_1\in\Bt$ such that
\[
U_1:= \underbrace{V_1+ \cdots +V_1}_{N+1} \subseteq U.
\]
By \ref{th:localbasis:3}, there is $V\in\Bt$ such that $\DD\cdot V\subseteq V_1$. Use \ref{th:localbasis:5} to pick
$\varepsilon>0$ so that $\lambda x\in V_1$ whenever $\abs{\lambda}\leq \varepsilon$, set $R=\min\enbrace{\varepsilon,N-\abs{\alpha}}$ and
consider $\mu\in D(\alpha,R)$ and $y\in V$. Since $\abs{\mu-\alpha}\leq R\leq\varepsilon$, $(\mu-\alpha)x\in V_1$. Furthermore,
\[
\abs{\dfrac{\mu}{N}}\leq \dfrac{\abs{\mu-\alpha}+\abs{\alpha}}{N}\leq \dfrac{R+\abs{\alpha}}{N}\leq 1,
\]
whence $\mu y/N\in V_1$. Summing up,
\[
\mu(x+y)
=\alpha x+(\mu-\alpha)x+\mu y
=\alpha x+(\mu-\alpha)x+N\dfrac{\mu}{N}y
\in \alpha x+U_1
\subseteq \alpha x + U. \qedhere
\]
\end{proof}

The most useful linear function spaces in Mathematical Analysis are the complete ones. We recall some terminology related to this notion. 

A \emph{directed set} is a partially ordered set $\Lambda$ such that for every $\lambda_1$, $\lambda_2\in\Lambda$ there is $\lambda\in\Lambda$ with $\lambda>\lambda_j$, $j\in\enbrace{1,2}$. Given directed sets $\Lambda_1$, $\Lambda_2$, a map $\sigma\colon \Lambda_1\to \Lambda_2$ is said to \emph{increase to infinity} if it is increasing and for every $\lambda_2\in\Lambda_2$ there is $\lambda_1\in\Lambda_1$ with $\sigma(\lambda_1)>\lambda_2$. A \emph{net} in a space $\XX$ is a family $(x_\lambda)_{\lambda\in\Lambda}$ in $\XX$ for some directed set $\Lambda$. 

A topological vector space $\XX$ is said to be \emph{complete} if every Cauchy net in $\XX$ converges. If we restrict ourselves to first-countable spaces, completeness depends entirely on Cauchy sequences. In fact, a first-countable topological vector space $\XX$ is complete if and only if every Cauchy sequence has a convergent subsequence.
An \emph{$F$-space} is a first-countable complete topological vector space $\XX$. If this topology is locally convex, we say that $\XX$ is a \emph{Fr\'echet space}.

It is known \cite{KPR1984} that a topological vector space is first-countable if and only if it is metrizable, in which case, the distance can be defined through an $F$-norm. Recall that an $F$-norm on a vector space space $\XX$ is a subadditive map
\[
\norm{\cdot}\colon \XX\to[0,\infty)
\]
such that $\norm{x}>0$ if $x\not=0$, and
\begin{enumerate}[label=(F.\arabic*)]
\item $\sup_{\abs{\lambda}\le 1} \norm{\lambda x}\le\norm{x}$ and $\lim_{\lambda\to 0} \norm{\lambda x}=0$ for all $x\in\XX$.
\end{enumerate}
If we replace subadditivity with the weak assumption that
\begin{enumerate}[label=(F.\arabic*),resume]
\item\label{it:quasiaditive} there is a constant $\kappa$ such that $\norm{x+y}\le \kappa \enpar{\norm{x} +\norm{y}}$ for all $x$, $y\in\XX$,
\end{enumerate}
we call $\norm{\cdot}$ a \emph{quasi-$F$-norm}, and we call \emph{modulus of concavity} of $\norm{\cdot}$ to the optimal constant in \ref{it:quasiaditive}. We
can use Theorem~\ref{th:localbasis} to prove that any quasi-$F$-norm $\norm{\cdot}$ induces on $\XX$ a first-countable vector topology for which
\[
\enbrace{x\in \XX \colon \norm{x}\le \varepsilon}, \quad \varepsilon>0,
\]
is a neighbourhood basis of the origin. This topology coincides with the one associated with the quasi-distance
\[
d_F\colon\XX\times \XX \to [0,\infty), \quad (x,y)\mapsto d_F(x,y):=\norm{x-y},
\]
Besides, if $F$ is an $F$-norm then $d_F$ is a distance.

The completeness of spaces defined from $F$-norms or quasi-$F$-norms is seldom proven in practice. When checked, the authors usually invoke criteria that take advantage of the particular geometry of the quasi-$F$-norm. Suppose, for instance, that $\norm{\cdot}$ is a norm, that is, it is subadditive and
\begin{enumerate}[label=(F.\arabic*), resume]
\item\label{it:hom} $ \norm{\lambda x}=\abs{\lambda} \norm{x}$ for all $x\in\XX$ and $\lambda\in\FF$.
\end{enumerate}
In this particular case, it is known that $(\XX,\norm{\cdot})$ is complete, i.e., a Banach space, if and only if any norm-convergent series converges. This criterion can be generalized to quasi-norms. Recall that a \emph{quasi-norm} is a quasi-$F$-norm that satisfies \ref{it:hom}. Given $0<p\le 1$, a $p$-norm is quasi-norm $\norm{\cdot}$ such that $\norm{\cdot}^p$ is subadditive. A topological vector topology $\XX$ is quasi-normable, that is, there is a quasi-norm on $\XX$ that induces the vector topology, if and only if it is locally bounded, that is, there is a bounded neighbourhood of the origin. On the one hand, any quasi-norm is equivalent to a $p$-norm for some $0<p\le 1$ by the Aoki--Rolewicz theorem \cites{Aoki1942,Rolewicz1957}.  On the other hand, a vector space $\XX$ equipped with a $p$-norm $\norm{\cdot}$ is complete if and only if a series $\sum_{n=1}^\infty x_n$ in $\XX$ converges provided that
\[
\sum_{n=1}^\infty \norm{x_n}^p <\infty.
\]
A small drawback of this criterion may be the need to find the index $p$ with which we apply it. But, above all, we must emphasize that there are spaces that are not locally bounded so they can not be equipped with a quasi-norm. For instance, this is the case of the space of measurable functions over the unit interval $[0,1]$ endowed with the topology of the convergence in measure, the space of holomorphic functions on the unit disc with the topology of uniform convergence of compact sets, and certain Orlicz spaces constructed from a nonconvex function. When dealing with such spaces, in the lack of a general criterion, we have to prove completeness using methods built ad hoc.

In this note, we give a simple characterization of completeness that works for any metrizable vector space. Our criterion for completeness does not depend on the existence of a metric on the vector space but relies on handling suitable neighbourhoods of the origin. Specifically, we will use the following concept.

\begin{definition}
We say that a sequence $\enpar{V_n}_{n=1}^\infty$ of subsets of a vector space is \emph{strongly nested} if
\[
V_{n+1}+V_{n+1}\subseteq V_n, \quad n\in\NN.
\]
\end{definition}

We remark that if $\enpar{V_n}_{n=1}^\infty$ is a strongly nested sequence of neighborhoods of the origin, then, since $0\in V_n$ for all $n\in\NN$, $V_{n+1}\subseteq V_n$ for all $n\in\NN$. Moreover, $\enbrace{0}\subsetneq V_n$ for all $n\in\NN$ provided $\XX$ is not the trivial space. Consequently, $V_{n+1}\subsetneq V_n$ for all $n\in\NN$.

\begin{lemma}\label{lem:SNa}
Let $\XX$ be a complete first-countable topological vector space and $\enpar{V_n}_{n=0}^\infty$ be a strongly nested set constituting a neighbourhood basis of the origin. Assume that $V_0$ is closed.
\begin{enumerate}[label=(\roman*)]
\item\label{it:SNa:1} Let $\enpar{y_n}_{n=1}^\infty$ be a sequence in $\XX$ with $y_1\in V_1$ and $y_{n+1}-y_n\in V_{n+1}$ for all $n\in\NN$. Then, $\enpar{y_n}_{n=1}^\infty$ converges, and $\lim_n y_n\in V_0$.
\item\label{it:SNa:2} Let $\enpar{x_n}_{n=1}^\infty$ be a sequence in $\XX$ with $x_n\in V_n$ for all $n\in\NN$. Then, the series $\sum_{n=1}^\infty x_n$ converges, and its sum belongs to $V_0$.
\end{enumerate}
\end{lemma}

\begin{proof}
To prove the \ref{it:SNa:1} we shall show by induction on $k$ that
\begin{enumerate}[label=(A)]
\item\label{it:SN} $y_{n+k}-y_n\in V_n$ for all $n\in\NN$ and all $k\in\ZZ$ nonnegative.
\end{enumerate}
As the origin lies in every neighborhood of the origin, the base case is trivial. Assume that the statement holds for a given $k$. We then have
\[
y_{n+k+1}-y_n=y_{n+k+1}-y_{n+1}+y_{n+1}-y_n\in V_{n+1}+V_{n+1} \subseteq V_n.
\]

We infer from \ref{it:SN} that $\enpar{y_n}_{n=1}^\infty$ is a Cauchy sequence that converges to a vector $y\in\XX$. Furthermore, $y-y_1\in \overline{V_1}$ and, therefore,
\[
y=y-y_1+y_1\in \overline{V_1}+V_1\subseteq \overline{V_1}+\overline{V_1} \subseteq \overline{V_1+V_1}\subseteq \overline{V_0}=V_0.
\]

To prove \ref{it:SNa:2}, just apply \ref{it:SNa:1} to the sequence $\enpar{y_n}_{n=1}^\infty$ given by $y_n=\sum_{j=1}^n x_j$.
\end{proof}

\begin{lemma}\label{lem:SNb}
Let $\XX$ be a first-countable topological vector space and $\Ct$ be a subset of the set of all convergent sequences. Suppose that $(x_0+x_n)_{n=1}^\infty\in\Ct$ for all $x_0\in\XX$ and  $(x_n)_{n=1}^\infty\in\Ct$. Assume that there exists a sequence $\enpar{V_n}_{n=1}^\infty$ of neighborhoods of the origin with the following property:
\begin{itemize}
\item For every sequence $\enpar{x_n}_{n=1}^\infty$ in $\XX$ with $x_n\in V_n$ for all $n\in\NN$, the sequence $\enpar{\sum_{k=1}^n x_k}_{n=1}^\infty$ of partial sums belongs to $\Ct$.
\end{itemize}
Then, $\XX$ is complete, and every convergent sequence has a subsequence belonging to $\Ct$.
\end{lemma}

\begin{proof}
Given a Cauchy sequence $(x_k)_{k=1}^\infty$ in $\XX$, we choose an increasing sequence $(k_n)_{n=1}^\infty$ of natural numbers such that $x_k-x_j\in V_{n+1}$ for all $j$, $k\in[k_n,\infty)\cap\ZZ$. By assumption, the sequence of partial-sums of the telescopic series $\sum_{n=1}^\infty x_{k_{n+1}} -x_{k_{n}}$ belongs to $\Ct$ . Hence, the subsequence $(x_{k_n})_{n=1}^\infty$ of  $(x_k)_{k=1}^\infty$ belongs to $\Ct$ and thus $(x_k)_{k=1}^\infty$ converges.
\end{proof}

Combining Lemma~\ref{lem:SNa} with Lemma~\ref{lem:SNb} yields the characterization of complete topological vector spaces that we record below.

\begin{theorem}\label{thm:SN}
Suppose that a strongly nested sequence $\enpar{V_n}_{n=1}^\infty$ in a vector space $\XX$ is a neighborhood basis of the origin for a vector topology. Then, $\XX$ is complete if and only if the series $\sum_{n=1}^\infty x_n$ converges provided $x_n\in V_n$ for all $n\in\NN$.
\end{theorem}

\section{Applications}\label{sect:examples}\noindent
We start this section by showing how Theorem~\ref{th:localbasis} and Theorem~\ref{thm:SN} can be used to prove that the space of measurable functions endowed with the topology of the convergence in measure is an $F$-space. For broader applicability, we will conduct the study within the framework of quasi-Banach-valued functions. Given a measure space $(\Omega, \Sigma, \mu)$ and a quasi-Banach space $\XX$ equipped with the Borel $\sigma$-algebra, a function $f\colon\Omega\to\XX$ is said to be \emph{strongly measurable} if it is measurable and there is a separable subspace $\YY$ of $\XX$ such that $f(\omega)\in\YY$ $\mu$-a.e.\@ $\omega\in\Omega$. It is known \cite{DiestelUhl1977} that $f$ is strongly measurable if and only if there is a sequence $(s_n)_{n=1}^\infty$ of simple functions such that $\lim_n s_n(\omega)=f(\omega)$ $\mu$-a.e.\@ $\omega\in\Omega$, in which case we can choose $(s_n)_{n=1}^\infty$ so that the sequence $(\norm{s_n})_{n=1}^\infty$ of scalar-valued functions is nondecreasing. We denote by \[
L(\mu,\XX)
\]
the vector space of all strongly measurable functions from $\Omega$ to $\XX$, identifying, as customary, functions that only differ on a null set.

We denote by $\Sigma(\mu)$ the set of all measurable sets with finite measure.

Given $E\in\Sigma$, we denote by $(E,\Sigma|_E,\mu|_E)$ the measure space obtained by restricting $(\Omega,\Sigma,\mu)$ to $E$.

Let $p\in(0,1]$. A subset $V$ of a vector space is said to be \emph{$p$-convex} if $su+tv\in V$ for all $u$, $v\in V$ and all $s$, $t\in \FF$ with $\abs{s}^p+\abs{t}^p \le 1$. A topological vector space is said to be \emph{locally $p$-convex} if there is a neighbourhood basis at the origin consisting of $p$-convex sets. Any locally bounded topological vector space $\XX$ is locally $p$-convex for some $p$. In fact, $\XX$ is locally $p$-convex if and only if $\XX$ can be equipped with a $p$-norm. Despite quasi-norms may not be continuous maps relative to the topology they induce \cite{Hyers1939}, $p$-norms are. Hence, any quasi-Banach space can be endowed with a continuous quasi-norm.

\begin{example}
Let $(\Omega,\Sigma,\mu)$ be measure space and $\XX$ be a quasi-Banach space endowed with a continuous quasi-norm and the Borel $\sigma$-algebra.
Given $f\in L(\mu,\XX)$ and $t\in(0,\infty)$, we consider the set
\begin{equation}\label{eq:levelset}
\Omega_{f,t}=\enbrace{\omega\in\Omega \colon \norm{f(\omega)}> t}.
\end{equation}
Clearly,
\begin{equation}\label{eq:DF2}
\Omega_{\lambda f,t} = \Omega_{f,t/\abs{\lambda}}, \quad f\in L(\mu,\XX), \, t>0,\, \lambda\in\FF\setminus\{0\}.
\end{equation}
If $\kappa$ is the modulus of concavity of $\XX$, then
\begin{equation}\label{eq:DF1}
\Omega_{f+g,t}\subseteq \Omega_{f,t/(2\kappa)}+\Omega_{g,t/(2\kappa)}, \quad f,\, g\in L(\mu,\XX), \, t\in(0,\infty).
\end{equation}
Now, define
\begin{equation}\label{eq:NeiL0}
    V_{E,\delta,t}=\enbrace{f\in L(\mu,\XX) \tq \mu(E\cap\Omega_{f,t}) < \delta}, \, E\in\Sigma, \, \delta, \, t\in(0,\infty).
\end{equation}
If two sets $A$, $E\in\Sigma$ satisfy $\mu(A\setminus E)\le\varepsilon$ then
the mere definition gives
\begin{equation}\label{eq:AlmostInclusion}
V_{E,\delta,t} \subseteq V_{A,\delta+\varepsilon,t}, \quad \delta,t,\varepsilon\in(0,\infty).
\end{equation}
By \eqref{eq:DF2},
\begin{equation*}
\lambda \cdot V_{E,\delta,t}= V_{E,\delta,\abs{\lambda} t} , \quad E\in\Sigma,\, t>0, \, \lambda\in\FF\setminus\{0\}.
\end{equation*}
In turn, by \eqref{eq:DF1},
\begin{equation}\label{eq:DF4}
    V_{E,\delta/2, t/(2\kappa)}+V_{E,\delta/2, t/(2\kappa)} \subseteq V_{E,\delta,t}, \quad E\in\Sigma,\, \delta, \, t\in(0,\infty).
\end{equation}
Since $\Omega_{f,t}$ decreases to $\emptyset$ as $t$ increases,
\[
    V_{E,t_1,\delta_1} \subseteq V_{E,t,\delta}, \quad E\in\Sigma, \, 0<t_1\le t<\infty,\, 0<\delta_1\le\delta<\infty,
\]
and furthermore,
\begin{align*}
    \bigcup_{t>0} V_{E,\delta,t}=&L(\mu,E,\XX)\\
    :=&\enbrace{f\in L(\mu,\XX) \colon \mu(E\cap\Omega_{f,t})< \infty \mbox{ for some } t>0 }
\end{align*}
for all $E\in\Sigma$ and $\delta>0$. By \eqref{eq:DF4}, $L(\mu,E,\XX)$ is a vector space.

When studying these kinds of spaces, it will be useful to consider certain families of sets. We say that a family $\Ft\subseteq\Sigma$  is \emph{directed to} $\Omega_0\in\Sigma$  if 
\begin{itemize}
    \item for every $A$, $B\in\Ft$ there is $D\in\Ft$ with $A\cup B\subseteq D$;
    \item $\mu(A\setminus \Omega_0)=0$ for all $A\in\Ft$; and 
    \item there is a countable family $\Gt\subseteq\Ft$ such that $\mu(\Omega_0\setminus\cup_{A\in\Gt} A)=0$.
\end{itemize}
Since $\Omega_{f,t}$ increases to $\enbrace{\omega\in\Omega \colon f(\omega)\not=0}$ as $t$ decreases to $0$,
\[
\bigcap_{\substack{t,\delta>0\\ E\in\Ft}} V_{E,\delta,t}=\enbrace{f\in L(\mu,\XX) \colon f(\omega)=0 \, \mu\mbox{-a.e.\@ } \omega\in\Omega_0}
\]
as long as $\Ft$  is \emph{directed to} $\Omega_0$.  We will infer from Theorem~\ref{th:localbasis} that if $\Ft$ is directed to $\Omega$ then
\begin{equation*}
\Vt:=\enbrace{V_{E,\delta,t} \tq E\in\Ft,\,\delta,t\in(0,\infty)}
\end{equation*}
is a local basis at the origin for a vector topology on
\[
    L(\mu,\Ft,\XX):=\bigcap_{E\in\Ft} L(\mu,E,\XX).
\]
Most of the conditions required by the theorem are easy to prove, only $\Vt$ satisfying\ref{th:localbasis:5} is in doubt. Given $E\in\Ft$, $\delta$, $t\in(0,\infty)$ and $f\in L(\mu,\Ft,\XX)$, there is $s\in(0,\infty)$ such that $f\in V_{E,\delta,s}$. Now, set
$\varepsilon=t/s$. If $\abs{\lambda}\le \varepsilon$, then $\lambda f\in V_{E,\delta,\abs{\lambda} s} \subseteq V_{E,\delta,t}$.

We denote by $L_0(\mu,\Ft,\XX)$ the above constructed topological vector space over $L(\mu,\Ft,\XX)$. Note that the convergence in $L_0(\mu,\Ft,\XX)$ is the convergence in measure on sets of $\Ft$.
The most important space constructed this way corresponds to the case when $\Ft=\{\Omega\}$. This is the space $L_0(\mu,\XX)$ associated with convergence in measure. If $\Sigma(\mu)$ is directed to $\Omega$, that is, $\mu$ is $\sigma$-finite, we must also be aware of the usefulness of the space 
\[
    L_{0,f}(\mu,\XX):=L_0(\mu,\Sigma(\mu),\XX)
\]
associated with the convergence in measure on finite-measure sets.
While $L_0(\mu,\XX)$ is a topological vector space over  
\[
    L(\mu,\Omega,\XX)=\enbrace{f\in L(\mu,\XX) \colon \mu(\Omega_{f,t})< \infty \mbox{ for some } t>0 },
\]
$L_{0,f}(\mu,\XX)$ is a topological vector space over $L(\mu,\XX)$. These two spaces coincide when $\mu$ is a finite measure. In contrast, if $\mu$ is $\sigma$-finite but infinite there exists $f\in L(\mu,\XX) \setminus L(\mu,\Omega,\XX)$. Indeed, we can choose
\[
    f=\sum_{n=1}^\infty \lambda_n \chi_{A_n},
\]
where $(\lambda_n)_{n=1}^\infty$ is an unbounded sequence in $\XX$ and $(A_n)_{n=1}^\infty$ is partition of $\Omega$ into finite-measure sets.

In our study, it will be useful to consider the following relationship between families of sets. Given $\Ft_1$ and $\Ft_2$  subsets of $\Sigma$, we say that $\Ft_1\prec \Ft_2$ if for all $A\in\Ft_1$ and $\varepsilon>0$ there is $B\in\Ft_2$ such that $\mu(A\setminus B)\le  \varepsilon$. This definition immediately gives that if $\Ft$ is directed to $\Omega$ then
\[
    \Sigma(\mu) \prec \Ft \prec \enbrace{\Omega}.
\]
Some elementary properties of the spaces $L_0(\mu,\Ft,\XX)$ of measurable functions are recorded below.

\begin{lemma}\label{lem:spacerelation}
    Let $(\Omega,\Sigma,\mu)$ be a measure space. Let $\Ft$ and $\Ft_0$ be subsets of $\Sigma$ directed to $\Omega$.
    \begin{enumerate}[label=(\roman*)]
        \item\label{it:newa}  $L_0(\mu,\Ft,\XX) \subseteq L_0(\mu,\Ft_0,\XX)$ continuously if and only if $\Ft_0\prec \Ft$.
    
        \item\label{it:newc} $L_0(\mu,\Ft_0,\XX) = L_0(\mu,\Ft,\XX)$ with the same topology if and only if $\Ft_0\prec \Ft\prec \Ft_0$. 
    
        \item\label{it:newb} $L_0(\mu,\Ft,\XX)$ is first-countable, that is, metrizable, if and only if there is a non-decreasing sequence $(A_n)_{n=1}^\infty$ in $\Ft$
        such that 
        \[
            \Ft\prec \Ft_0:=\enbrace{A_n \colon n\in\NN}.
        \]
        If this is the case,   $\Ft_0$ is directed to $\Omega$,  and, given $(\varepsilon_n)_{n=1}^\infty$ and $(t_n)_{n=1}^\infty$ in $(0,\infty)$ with $\lim_n \varepsilon_n=\lim_n t_n=0$, 
        \[
            \enbrace{ V_{A_n,\varepsilon_n, t_n} \colon n\in\NN}
        \] 
        is a countable neighbourhood basis of the origin of $L_0(\mu,\Ft,\XX) = L_0(\mu,\Ft_1,\XX)$.
    
        \item\label{it:newd} Given $E\in\Sigma$, the mapping $f\mapsto f|_E$ defines a linear bounded projection from $L_0(\mu,\Ft,\XX)$ onto $L_0(\mu|_E, \Ft|_E,\XX)$, where 
        \[
            \Ft|_E=\{A\cap E \colon A\in\Ft\}. 
        \]
    \end{enumerate}
\end{lemma}

\begin{proof}
The `if' part of \ref{it:newa} follows from \eqref{eq:AlmostInclusion}. To prove the `only if' part, we pick $A\in\Ft_0$ and $\varepsilon\in(0,\infty)$. Let $B\in\Ft$ and $\delta$, $t\in(0,\infty)$ be  such that $V_{B,\delta,t} \subseteq V_{A,\varepsilon,1}$ and pick $x\in\XX$ with $\norm{x}\ge 1$. Since $f:=x\chi_{A\setminus B}\in V_{B,\delta,t}$, 
\[
\mu(A\setminus B)=\mu\enpar{ \enbrace{\omega\in A  \colon \norm{f(\omega)} >1 }}<\varepsilon.
\]

Condition \ref{it:newc} is a direct consequence of \ref{it:newa}. To prove \ref{it:newb}, assume that the space $L_0(\mu,\Ft,\XX)$ is first-countable. In that case, there exists a non-decreasing sequence $(A_n)_{n=1}^\infty\subseteq \Ft$ such that $L_0(\mu,\Ft,\XX)$ has a neighbourhood basis of the origin consisting of sets $V_{A_n,\delta,t}$ with $n\in\NN$.
Hence, if we denote $\Ft_0=\enbrace{A_n \tq n\in\NN}$, $L_0(\mu,\Ft_0,\XX)\subseteq L_0(\mu,\Ft,\XX)$ continuously and consequently $\Ft\prec\Ft_0$. The remaining requirements follow readily. Finally, \ref{it:newd} is easy to prove.
\end{proof}

We next show that $L_0(\mu,\Ft,\XX)$ is complete. To begin with, consider the case when $\Ft$ is a singleton, that is $L_0(\mu,\Ft,\XX)=L_0(\mu,\XX)$.  We know from Lemma~\ref{lem:spacerelation} that this space is first-countable so, aiming to apply Lemma~\ref{lem:SNb}, we choose a sequence $(\varepsilon_n)_{n=1}^\infty$ in $(0,\infty)$ with $\sum_{n=1}^\infty \varepsilon_n<\infty$, and $t_n=(2\kappa)^{-n}$ for all $n\in\NN$. Set
\[
\delta_n=\sum_{j=n+1}^\infty \varepsilon_n, \quad\, W_n=\enbrace{x\in\XX \colon \norm{x} \le t_n}, \quad n\in\NN.
\]
Note that $W_{n+1}+W_{n+1}\subseteq W_n$ for all $n\in\NN$, so $(W_n)_{n=1}^\infty$ is a strongly nested local basis of of $\XX$. Now, let $(f_n)_{n=1}^\infty$ in $L(\mu,\XX)$ be such that $f_n\in V_{\Omega,\varepsilon_n, t_n}$ for all $n\in\NN$. We will show that $(V_{\Omega,\varepsilon_n, t_n})_{n=1}^\infty$ is the sequence of neighborhoods required by Lemma~\ref{lem:SNb}.
Set
\[
A_j= \cup_{n=j+1}^\infty \Omega_{f_n, t_n },\, \Omega_j=\Omega\setminus A_j, \quad j\in \NN.
\]
We have $\mu(A_j) < \delta_j$ for all $j\in \NN$. Consequently,
\[
A=\bigcap_{j=1}^\infty A_j.
\]
is a null set. Note that $\Omega\setminus A=\cup_{j=1} \Omega_j$ and that $\Omega_j$ consists of all $\omega\in\Omega$ such that $f_n(\omega)\in W_n$ provided that $n\ge j+1$. Applying Lemma~\ref{lem:SNa} in the quasi-Banach space $\XX$, the series $\sum_{n=1}^\infty f_n(\omega)$ converges for all $\omega\in \Omega\setminus A$, and $\sum_{n=j+1}^\infty f_n(\omega) \in W_{j}$ provided that $\omega\in \Omega_j$. Define the function
\begin{align*}
    f:\Omega&\to\XX \\
    \omega&\mapsto 
    \begin{cases}
        \sum_{n=1}^\infty f_n(\omega) & \text{if $\omega\in\Omega\setminus A$}\\
        0 & \text{otherwise}
    \end{cases}
\end{align*}
and set
$g_j=f-\sum_{n=1}^j f_j$ for all $j\in\NN$. We have
\[
    \Omega_{g_j,t_j} \subseteq A_j, \quad j\in\NN.
\]
Therefore,
\[
g_j\in V_{\Omega,\delta_j, t_j } \subseteq V_{\Omega,\delta_k, t_k } , \quad j,k\in \NN, \, k\le j.
\]
Since $(V_{\Omega,\varepsilon_k, t_k})_{k=1}^\infty$ is a neighbourhood basis at the origin, $\lim_j g_j=0$ and the space is complete.

In general, we pick a non-decreasing sequence $(A_n)_{n=1}^\infty$ in $\Ft$ with $\mu\enpar{\Omega\setminus\cup_{n=1}^\infty A_n}=0$. Let $(f_\lambda)_{\lambda\in \Lambda}$ be a Cauchy net in $\YY:=L_0(\mu,\Ft,\XX)$. Let $E\in\Ft$. Since the mapping $f\mapsto f|_{E}$ defines a continuous operator from $\YY$ into $L_0(\mu|_{E},\XX)$, there is  a strongly measurable function $f_E\colon E\to \XX$ such that $\lim_\lambda f_\lambda|_E=f_E$ in $L_0(\mu|_E,\XX)$. By the uniqueness of the limit in $L_0(\mu|_{A_n},\XX)$, $n\in\NN$, there is $f\in L_0(\mu,\XX)$ such that $f(\omega)=f_{A_n}(\omega)$ $\mu$-a.e.\@ $\omega\in A_n$ for all $n\in\NN$. Now, by the uniqueness of the limit in $L_0(\mu|_{E\cap A_n},\XX)$, $n\in\NN$, $f_E=f$ $\mu$-a.e.\@ $\omega\in E$ for all $E\in\Ft$. We infer that $\lim_\lambda f_\lambda=f$ in $\YY$. 

\medskip 

By Egoroff's Theorem, if a sequence in $L(\mu,\XX)$ converges to $f\in L(\mu,\XX)$ $\mu$-a.e., then it also converges to $f$ in $L_{0,f}(\mu,\XX)$. Let us show that a partial converse holds as long as $\mu$ is $\sigma$-finite. In our application of Lemma~\ref{lem:SNb}, its hypotheses are satisfied with $\Ct$ the set of all sequences converging both in measure and almost everywhere. We infer that  any sequence converging in $L_0(\mu,\XX)$ has a subsequence converging $\mu$-a.e.\@ $\omega\in\Omega$. A standard application of the diagonal Cantor method implies this result to be true for $L_0(\mu,\Ft,\XX)$ where $\Ft$ is an arbitrary set directed to $\Omega$. In particular, if $\mu$ is $\sigma$-finite and a sequence $(f_n)_{n=1}^\infty$ converges to $f$ in $L_{0,f}(\mu,\XX)$ then a subsequence of $(f_n)_{n=1}^\infty$ converges to $f$ $\mu$-a.e. 

\medskip
We close this example by noticing that if $\mu$ is $\sigma$-finite and  $(A_n)_{n=1}^\infty$ is a nondecreasing sequence in $\Sigma(\mu)$ with  $\cup_{n=1}^\infty A_n=\Omega$, then
\[
\Sigma(\mu)\prec \Ft_0:=\enbrace{A_n \colon n\in\NN}.
\]
Consequently, $L_{0,f}(\mu,\XX)=L_0(\mu,\Ft_0,\XX)$. As a result, $L_{0,f}(\mu,\XX)$ is metrizable. In contrast, for any infinite measure $\mu$ there is $\Ft$ directed to $\Omega$ such that $L_0(\mu,\Ft,\XX)$ is not metrizable. Pick a partition $(\Omega_n)_{n=1}^\infty$ of $\Omega$ such that $\mu(\Omega_n)=\infty$ for all $n\in\NN$. 
Let $\Ft$ be the set of all $A\in\Sigma$ for which there is $k\in\NN$ such that $\mu(A\setminus \Omega_n)<\infty$ whenever $n>k$. Assume by contradiction that $L_0(\mu,\Ft,\XX)$ is metrizable. Then, there is a nondecreasing sequence $(B_j)_{j=1}^\infty$ in $\Ft$ such that 
\[
\Ft\prec \enbrace{B_j \colon j\in\NN}.
\]
Pick an increasing sequence $(n_j)_{j=1}^\infty$ such that $\mu(B_j\setminus \Omega_n)<\infty$ for all $j$, $n\in\NN$ with $n>n_j$. 
Since for each $n\in\NN$
\[
\Nt_n=\{j \in\NN \colon n_j <n\}
\]
is finite, $E_n:= \cup_{j\in\Nt_n} (B_j\cap \Omega_n)$ is a finite-measure set. Hence, there is $D_n\in\Sigma$ with $D_n\subseteq \Omega_n$, $1\le \mu(D_n)<\infty$ and 
\[
D_n\cap E_n= D_n \cap \enpar{\cup_{j\in\Nt_n} B_j}  =\emptyset.
\]
We have $D:=\cup_{n=1}^\infty D_n\in\Ft$ and, for all $j\in\NN$,
\[
\mu\enpar{D\setminus B_j}=\sum_{n=1}^\infty \mu\enpar{D_n \setminus B_j}\ge \sum_{n=1+n_j}^\infty \mu\enpar{D_n \setminus B_j}
=\sum_{n=1+n_j}^\infty \mu\enpar{D_n}.
\]
In conclusion, $\mu\enpar{D\setminus B_j}=\infty$. This absurdity proves that  $L_0(\mu,\Ft,\XX)$ is not metrizable.
\end{example}

\begin{example}
We propose a general method for building spaces of measurable functions. Given a $\sigma$-finite measure space $(\Omega,\Sigma,\mu)$, we denote by $L^+(\mu)$ the cone of $L(\mu)$ consisting of all measurable functions with values in $[0,\infty]$. A \emph{gauge} will be a map $\rho\colon L^+(\mu) \to [0,\infty]$ such that
\begin{enumerate}[label=(F.\arabic*),widest=6,series=fqn]
\item\label{FQN:M} if $f\le g$ $\mu$-a.e., then $\rho(f)\le \rho(g)$.
\end{enumerate}
Let $J$ be the interval consisting of all $t\in(0,\infty)$ for which there is a constant $C\in(0,\infty)$ such that
\begin{equation}\label{eq:pseuso:hom}
\rho(t f)\le C \rho(f), \quad f\in L^+(\mu).
\end{equation}
Given $t\in J$ we denote by $\Delta(t)$ the optimal constant $C\in(0,\infty)$ in \eqref{eq:pseuso:hom}, and we call the function $\Delta\colon J\to(0,\infty]$ the \emph{homogeneity function of $\rho$}. Clearly $(0,1]\subseteq J$ and $\Delta(t)\le 1$ for all $t\in(0,1]$. If $s$, $t\in J$, then $st\in J$, and $\Delta(st)\le \Delta(s)\Delta(t)$. Hence, either $J=(0,1]$ or $J=(0,\infty)$. A gauge whose homogeneity function is defined in $(0,\infty)$ is said to be \emph{pseudo-homogeneous at infinity}. In turn, if
\[
\lim_{t\to 0^+} \Delta(t)=0.
\]
we say that the gauge $\rho$ is \emph{pseudo-homogeneous at the origin}. Note that if $\rho$ is pseudo-homogeneous at the origin, then $\Delta(t)<1$ for all $t\in(0,1)$. The other way around, if there is $t\in(0,1)$ with $\Delta(t)<1$ then $\rho$ is pseudo-homogeneous at the origin. We call \emph{homogeneous} those gauges whose homogeneity function is the identity map on $(0,\infty)$. If $\rho$ is homogeneous, then
\[
\rho(t f) =t\rho(f), \quad t\in(0,\infty),\, f\in L^+(\mu).
\]

\medskip
A \emph{modular function norm} will be a gauge $\rho$ such that
\begin{enumerate}[label=(F.\arabic*),widest=6,series=fqn,resume]
\item\label{FQN:QSA} there are constants $\kt$ and $\rt$ such that
\[
\rho(f+g)\le\kt( \rho(\rt f)+\rho(\rt g)), \quad f,g\in L^+(\mu).
\]
\item\label{FQN:CM} for every $E\in\Sigma(\mu)$ there is $u_E\in(0,\infty)$ such that
\begin{itemize}
\item $\rho(u_E \chi_E)<\infty$, and
\item for every $\varepsilon>0$ there is $\delta>0$ such that $\mu(A)\le \varepsilon$ whenever $A\in\Sigma$ satisfies $A\subseteq E$ and $\rho(\chi_A)\le \delta$;
\end{itemize}
\item\label{FQN:Abs} $\lim_{t\to 0^+} \rho(tf)=0$ for all $f\in L^+(\mu)$ with $\rho(f)<\infty$; and
\item\label{FQN:RFatou} there is a constant $\ft\in[1,\infty)$ such that
\[
\rho(\lim_n f_n)\le \ft \lim_n \rho(f_n)
\]
for all non-decreasing sequences $(f_n)_{n=1}^\infty$ in $L_0^+(\mu)$.
\end{enumerate}
Condition~\ref{FQN:RFatou} is a Fatou-type property called \emph{rough Fatou property} (see \cite{AnsorenaBello2022}). We say that $\ft$ is a \emph{Fatou constant} for $\rho$ and if $\ft=1$ then $\rho$ has the \emph{Fatou property.} We call the pair $(\kt,\rt)$ in \ref{FQN:QSA} a \emph{convexity pair} for $\rho$. We can choose a convexity pair $(\kt,\rt)$ with $\rt=1$ if and only if $\rho$ is pseudo-homogeneous at infinity.
A generalization of the Aoki--Rolewicz Theorem (see \cite{AnsorenaBello2022}*{Proposition 1.1}), yields that if $(\kt,1)$ is an convexity pair for $\rho$ and $p\in(0,1]$ is given by  $2^{1/p-1}=\kt$, then there is a constant $C$ such that
\begin{equation}\label{eq:pAdd}  
\frac{1}{C}\, \rho\enpar{\sum_{j=1}^n f_j} \le M_{\rho} (\ff;p):=\enpar{\sum_{j=1}^n \rho^p(f_j)}^{1/p}
\end{equation}
for all $n\in\NN$ and all $\ff=(f_j)_{j=1}^n$ in $L^+(\mu)$. Consequently,  for any homogeneous gauge $\rho_0$ satisfying \ref{FQN:QSA} there is an homogeneous gauge $\rho$, a constant $C\in[0,\infty)$ and $0<p\le 1$ such that 
\[
\rho\le \rho_0 \le C\rho
\]
and 
\begin{equation*}
\rho\enpar{f + g}  \le  \enpar{\rho^p(f)+\rho^p(g)}^{1/p}, \quad f,g\in L^+(\mu).
\end{equation*}
Furthermore, if $\rho_0$ is a modular function norm, so is $\rho$.

Any pseudo-homogeneous gauge $\rho$ at the origin satisfies \ref{FQN:Abs}. Another condition ensuring this property is
\begin{enumerate}[label=(H)]
\item\label{FQN:AC} $\lim_n \rho(f_n)=0$ for all non-increasing sequences $(f_n)_{n=1}^\infty$ in $L^+(\mu)$ with $\lim_n f_n=0$.
\end{enumerate}
Gauges satisfying \ref{FQN:AC} are called \emph{absolutely continuous}.

Given $0<p<\infty$, we say that $\rho$ is \emph{lattice $p$-convex} with constant $C\in[1,\infty)$ if 
\[
\frac{1}{C} \, \rho\enpar{\enpar{\sum_{j=1}^n s_j^p f_j^p}^{1/p}} \le  M_\rho(\ff):= \max_{1\le j \le n} \rho(f_j)
\]
for all $n\in\NN$, all $(s_j)_{j=1}^n$ in $[0,\infty)$ with $\sum_{j=1}^n s_j^p\le 1$ and all $\ff=(f_j)_{j=1}^n$ in $L^+(\mu)$. If, for $n$,  $(s_j)_{j=1}^n$ and $\ff=(f_j)_{j=1}^n$ running over the same set,
\[
C\, \rho\enpar{\enpar{\sum_{j=1}^n s_j^p f_j^p}^{1/p}} \ge  m_\rho(\ff):= \min_{1\le j \le n} \rho(f_j),
\]
we say that $\rho$ is \emph{lattice $p$-concave} with constant $C$. Note that $\rho$ is lattice $p$-convex (resp., lattice $p$-concave) with constant $C$ if and only if the gauge
\[
f\mapsto \rho(f^{1/p}), \quad f\in L^+(\mu).
\]
is lattice $1$-convex (resp., $1$-concave) with constant $C$.

Given $f\in L^+(\mu)$, let $\Ft_p(f)$ be the set all $(f_j)_{j=1}^n$ in $L^+(\mu)$ such that  $f^p = \sum_{j=1}^n s_j^p f_j^p$ for some $(s_j)_{j=1}^n$ in $[0,\infty)$ with $\sum_{j=1}^n s_j^p=1$.  If $\rho$ is lattice $p$-convex with constant $C$, then the gauge $\rho_0$ given by
\[
\rho_0(f)= \inf\enbrace{ M_\rho(\ff) \colon \ff \in \Ft_p(f)  }, \quad f\in L^+(\mu),
\]
satisfies $\rho_0\le \rho\le  C \rho_0$ and is lattice $p$-convex with constant one. Similarly, if $\rho$ is lattice $p$-concave with constant $C$, then
the gauge $\rho_0$ given by
\[
\rho_0(f)= \sup\enbrace{ m_\rho(\ff) \colon \ff \in \Ft_p(f)  }, \quad f\in L^+(\mu),
\]
satisfies  $\rho_0\ge \rho \ge  C \rho_0$ and is $p$-concave with constant one. In both cases, if $\rho$ is homogeneous, so is $\rho_0$. If $\rho$ is homogeneous, then $\rho$ is lattice $p$-convex (resp., $p$-concave) with constant $C$ if and only if
\[
N_{\rho}(\ff;p) := \rho\enpar{\enpar{\sum_{j=1}^n  f_j^p}^{1/p}}\le C M_{\rho}(\ff;p)
\]
(resp., $M_{\rho}(\ff;p)\le C N_{\rho}(\ff;p)$) for all $n\in\NN$ and all $\ff=(f_j)_{j=1}^n$ in $L^+(\mu)$.

Note that, for a homogeneous gauge, lattice $1$-convexity is the same condition as that in \eqref{eq:pAdd}, while, if $0<p<1$,  lattice $p$-convexity is stronger than \eqref{eq:pAdd}.

A \emph{function quasi-norm} will be a homogeneous modular function norm. A \emph{function norm} will be a function quasi-norm for which $(1,1)$ is a convexity pair. When dealing with function norms, it is customary to replace \ref{FQN:CM} with the existence of $(C_E)_{E\in\Sigma(\mu)}$ in $(0,\infty)$ such that
\[
\rho(\chi_E)<\infty \,\, \text{and} \int_E f \, d\mu \le C_E \rho(f)
\]
for all $ E\in\Sigma(f)$ and $ f\in L^+(\mu)$ (see \cite{BennettSharpley1988}). Condition~\ref{FQN:CM} is a weaker continuity condition proposed by the authors of \cite{AnsorenaBello2022} to fit the particularities of locally non-convex spaces. This condition implies $\rho(u_E\chi_E)>0$ for all $E\in\Sigma(\mu)$ with $\mu(E)>0$. Indeed, if $0<\varepsilon<\mu(E)$ and $\delta$ is as in \ref{FQN:CM}, then $\rho(u_E\chi_E)>\delta$. We claim that any modular function norm $\rho$ over $(\Omega,\Sigma,\mu)$ has the following properties.
\begin{enumerate}[label=(G.\arabic*),widest=2]
\item\label{FQN:Z} $\rho(0)=0$.
\item\label{it:PB}  If $f>0$ $\mu$-a.e., then there is $0<t<\infty$ such that $\rho(tf)>0$. 
\item\label{it:PC} If $(f_n)_{n=1}^\infty$ is a sequence in $L^+(\mu)$ with $f_n(\omega)<\infty$ for all $n\in\NN$ and $\omega\in\Omega$, and
\[
\lim_{k} \rho\enpar{\sum_{n=k}^\infty f_n}=0,
\]
then $\sum_{n=1}^\infty f_n<\infty$ a.e.
\end{enumerate}
For every $E\in\Sigma(\mu)$, define $u_E$ as the constant given by \ref{FQN:CM}. \ref{FQN:Z} follows from combining \ref{FQN:CM} with $E=\emptyset$ and \ref{FQN:Abs}. To show \ref{it:PB}, we choose $E\in\Sigma(\mu)$ and $t\in(0,u_E]$ such that $ t \chi_E\leq f$ and $\mu(E)>0$. We have
\[
\rho\enpar{\frac{u_E f}{t}} \ge \rho(u_E \chi_E) >0.
\]
Finally, we set $E=\enbrace{ \omega\in\Omega \colon \sum_{n=1}^\infty f_n=\infty}$. Since
\[
u_E\chi_E\le \sum_{n=k}^\infty f_n, \quad k\in \NN,
\]
$\rho(u_E\chi_E)=0$. Therefore, $\mu(E)=0$ and \ref{it:PC} follows.

\medskip
We go on by studying two kinds of gauges associated with modular function norms. Firstly, given a modular function norm $\rho$, we consider the Luxemburg map
\[
\tilde{\rho}(f) =\inf\enbrace{t\in(0,\infty) \colon \rho(f/t) < 1},\quad f\in L^+(\mu).
\]
Note that this map is the Minkowski functional associated with the set $\enbrace{f\in L^+ \tq \rho(f)<1}$.
The maps $\rho$ and $\tilde{\rho}$ are related as follows.
\begin{enumerate}[label=(A.\arabic*),widest=3,series=a]
\item $\rho(f) < \Delta(t)$ provided that $\tilde{\rho}(f)<t$. 
\item $\tilde{\rho}(f)<t$ provided that $\rho(f)<1/\Delta(1/t)$.
\end{enumerate}
Consequently,
\begin{enumerate}[label=(A.\arabic*),widest=3,series=a,resume]
\item\label{it:IneqC} assuming that $\rho$ is pseudo-homogeneous at the origin, for every $t>0$ there is $s=s(t)>0$ such that $\tilde{\rho}(f)<t$ whenever $\rho(f)<s$, and, the other way around,
$\rho(f)<t$ whenever $\tilde{\rho}(f)<s$.
\end{enumerate}

We claim that if the modular function norm $\rho$ is pseudo-homogeneous at the origin, then $\tilde{\rho}$ is a function quasi-norm. Indeed, the mere definition gives the homogeneity of $\tilde{\rho}$. $\tilde{\rho}$ plainly inherits \ref{FQN:M} from $\rho$. By \ref{it:IneqC}, $\tilde{\rho}$ inherits \ref{FQN:CM} and \ref{FQN:Abs} as well. Let $(\kt,\rt)$ be a convexity pair for $\rho$. Pick $\tau\in(0,1)$ with $\Delta(\tau)\le 1/(2\kt)$. Let $f_1$, $f_2\in L^+(\mu)$. If $t>M:=\max\{\tilde{\rho}(\rt f_1) , \tilde{\rho}(\rt f_2)\}$, then
\[
\rho\enpar{ \tau \frac{f_1+f_2}{ t } } \le \frac{1}{2\kt} \rho\enpar{\frac{f_1}{t} + \frac{f_2}{t} }\le \frac{1}{2} \enpar{\rho\enpar{\frac{\rt f_1}{t}} + \rho\enpar{\frac{\rt f_2}{t} } }<1.
\]
Hence, $\tilde{\rho}(f_1+f_2) \le M/\tau$. Therefore, $(1/\tau,\rt)$ is a convexity pair for $\tilde{\rho}$. Let $\lambda\in(0,\infty)$ be such that $\Delta(1/\lambda)<1/\ft$ where $\ft$ is a Fatou constant for $\rho$. Let $(f_n)_{n=1}^\infty$ be a non-decreasing sequence in $L^+(\mu)$ that converges to $f$. Pick $t>\lim_n \tilde{\rho}(f_n)$. Since $\rho(f_n/t)<1$ for all $n\in\NN$, $\rho(f/t)\le \ft$. Hence, $\rho(f/(\lambda t))<1$ and thus, $\tilde{\rho}(f) \le \lambda \lim_n \tilde{\rho}(f_n)$. So, $\lambda$ is a Fatou constant for $\tilde{\rho}$. 

If $\rho$ is lattice $p$-convex or lattice $p$-concave for some $0<p\le 1$, so is $\tilde{\rho}$.

In conclusion, the Luxemburg gauge allows us to build a function quasi-norm from a modular function norm that is pseudo-homogeneous at the origin. However, this construction is not useful when considering modular function norms that are not pseudo-homogeneous at the origin. Our second construction fits any modular function norm and yields a modular function norm pseudo-homogeneous at infinity. Following \cite{MusielakOrlicz1959}, we build from a modular function norm $\rho$ the gauge $\overline{\rho}$ given by
\[
\overline{\rho}(f)=\inf\enbrace{ t\in (0,\infty) \colon \rho(f/t)<t }, \quad f\in L^+(\mu).
\]
We have the following facts.
\begin{enumerate}[label=(B.\arabic*),widest=3,series=b]
\item\label{it:IneqE} $\max\{\rho(f),\tilde{\rho}(f)\} \le \overline{\rho}(f)$ provided that $\overline{\rho}(f)<1$; and
\item\label{it:IneqD} $ \overline{\rho}(f)<t$ provided that $\rho(f)< t/\Delta(1/t)$.
\end{enumerate}
Let us prove that $\overline{\rho}$ is a modular function norm pseudo-homogeneous at infinity. Clearly, $\overline{\rho}$ inherits \ref{FQN:M} from $\rho$. By \ref{it:IneqE} and by \ref{it:IneqD}, $\overline{\rho}$ inherits \ref{FQN:CM} and \ref{FQN:Abs} as well. Let $(\kt,\rt)$ be a convexity pair for $\rho$. Set $\lambda=\max\{2\kt/\rt,1\}$.

Let $f_1$, $f_2\in L^+(g)$. If $t>M:=\rt \max\{\overline{\rho}(f_1) , \overline{\rho}(f_2)\}$, then
\[
\rho\enpar{ \frac{f_1+f_2}{ \lambda t } } \le \rho\enpar{\frac{f_1+f_2}{t} }\le \kt \enpar{\rho\enpar{\frac{\rt f_1}{t}} + \rho\enpar{\frac{\rt f_2}{t} } }< \frac{2\kt t}{\rt}\le \lambda t.
\]
Consequently, $\overline{\rho}(f_1+f_2)\le \lambda t$. Hence, $\overline{\rho}(f_1+f_2)\le \lambda M$. So, $(\lambda,1)$ is a convexity pair for $\overline{\rho}$.

Let $(f_n)_{n=1}^\infty$ be a non-decreasing sequence in $L^+(\mu)$ and $F$ be a Fatou constant for $\rho$. Set $f=\lim_n f_n$. Pick $t>s>\lim_n \overline{\rho}(f_n)$. Since $\rho(f_n/s)<s$ for all $n\in\NN$, $\rho(f/s)\le \ft s$. Consequently, $\rho(f/(\ft t))<\ft t$. Therefore, $\overline{\rho}(f) \le \ft t$. Hence, $\ft$ is a Fatou constant for $\overline{\rho}$. 

 As well as $\widetilde{\rho}$, the gauge $\overline{\rho}$ inherits  lattice convexity and lattice concavity from $\rho$.

\medskip
We now construct vector-valued spaces of measurable functions from modular function norms. Let $\XX$ be a quasi-Banach space endowed with a continuous quasi-norm $\norm{\cdot}$. The balls associated with $\rho$ and $\XX$ are the sets
\begin{equation*}
B_{\rho,\XX}(\varepsilon)=\enbrace{ f\in L(\mu,\XX) \colon \rho(\norm{f})<\varepsilon}, \quad \varepsilon>0.
\end{equation*}
We shall show the following.
\begin{enumerate}[label=(\roman*)]
\item $\enbrace{ u B_{\rho,\XX}(\varepsilon) \colon u, \varepsilon\in(0,\infty) }$ is a neighbourhood basis at the origin for a complete vector topology on
\[
L_\rho(\XX)=\enbrace{f\in L(\mu,\XX) \colon \rho(u\norm{f}) <\infty \mbox{ for some } u>0}.
\]
\item $L_\rho(\XX)=L_{\overline{\rho}}(\XX)$ with the same topology.
\item\label{eq:CEVTPC} In the case when $\rho$ is pseudo-homogeneous at infinity
\[
L_\rho(\XX)=\enbrace{f\in L(\mu,\XX) \colon \rho(\norm{f}) <\infty}
\]
and $\Bt_{\rho,\XX}:=\enbrace{B_{\rho,\XX}(\varepsilon) \colon \varepsilon\in(0,\infty)}$ is a neighbourhood basis at the origin for the topology on $L_\rho(\XX)$.
\item In the case when $\rho$ is pseudo-homogeneous at the origin $L_\rho(\XX)=L_{\tilde{\rho}}(\XX)$ with the same topology, and $L_{\tilde{\rho}}(\XX)$  endowed with the quasi-norm $\norm{\cdot}_0=\tilde{\rho}(\norm{\cdot})$ is a quasi-Banach space.
\item \label{eq:pconv}
Let $p\in(0,1]$. If $\rho$ is lattice $p$-convex and $\XX$ is a locally $p$-convex,  then $L_\rho(\XX)$ is locally $p$-convex.
\end{enumerate}

Note that any modular function norm $\rho$ satisfies
\begin{equation*}
B_{\overline{\rho},\XX}(\min\{u,\varepsilon\}) \subseteq u B_{\rho,\XX}(\varepsilon), \quad \varepsilon B_{\rho,\XX}(\varepsilon) \subseteq B_{\overline{\rho},\XX}(\varepsilon), \quad u, \varepsilon>0.
\end{equation*}
Besides, if the modular function norm $\rho$ is pseudo-homogeneous at infinity,
\[
B_{\rho,\XX}(\Delta(1/u) \varepsilon) \subseteq u B_{\rho,\XX}(\varepsilon),\quad u,\, \varepsilon>0.
\]
Finally, if the modular function norm $\rho$ is pseudo-homogeneous at the origin \ref{it:IneqC} gives $\tau\colon(0,\infty)\to(0,\infty)$ such that
\[
B_{\rho,\XX}(\tau(\varepsilon)) \subseteq B_{\tilde{\rho},\XX}(\varepsilon), \quad B_{\tilde{\rho},\XX}(\tau(\varepsilon)) \subseteq B_{\rho,\XX}(\varepsilon), \quad \varepsilon>0. 
\]
Hence, since proving \ref{eq:pconv} is a routine checking,
it suffices to prove the latter assertion in \ref{eq:CEVTPC}. To that end, we choose $\kt\in[1,\infty)$ such that $(\kt,1)$ is a convexity pair for $\rho$. Let $\kt_0$ be the modulus of concavity of $\norm{\cdot}$. We have
\begin{itemize}
\item $(B_{\rho,\XX}(\varepsilon))_{\varepsilon>0}$ increases as $\varepsilon$ increases;
\item if $\varepsilon>0$, $g\in B_{\rho,\XX}(\varepsilon)$, and $f\in L(\mu,\XX)$ satisfies $\norm{f}\le\norm{g}$, then $f\in B_{\rho,\XX}(\varepsilon)$; and
\item if $\varepsilon_1$, $\varepsilon_2\in(0,\infty)$, then
\[
B_{\rho,\XX}(\varepsilon_1)+ B_{\rho,\XX}(\varepsilon_2) \subseteq B_{\rho,\XX}(\kt\Delta(\kt_0) (\varepsilon_1 + \varepsilon_2)).
\]
\end{itemize}
These properties readily give conditions \ref{th:localbasis:1} and \ref{th:localbasis:2} and \ref{th:localbasis:3} in Theorem~\ref{th:localbasis}. In turn \ref{th:localbasis:5} follows from \ref{FQN:Abs}, and \ref{th:localbasis:4} follows from \ref{it:PB}.

Now, we will use Lemma~\ref{lem:SNb} to prove that $L_\rho(\XX)$ is complete. Pick $(t_n)_{n=1}^\infty$ in $(0,\infty)$ such that $\sum_{n=1}^\infty t_n<\infty$. Set
\[
\delta_n=\sum_{k=n}^\infty t_k, \quad \varepsilon_n= \kt^{-n} \enpar{\Delta(\kt_0^n)}^{-1} t_n,
\quad V_n=B_{\rho,\XX}(\varepsilon_n), \quad n\in \NN.
\]
Pick a sequence $(f_n)_{n=1}^\infty$ in $L_\rho(\XX)$ satisfying $f_n\in V_n$ for all $n\in\NN$. Then, for any $k$, $m\in\NN$ with $k\le m$,
\[
\rho\enpar{\sum_{n=k}^m \kt_0^n \norm{f_n}} \le \sum_{n=k}^m \kt^n \rho\enpar{\kt_0^n \norm{f_n}}
\le \sum_{n=k}^m \kt^n \Delta(\kt_0^n) \rho\enpar{\norm{f_n}} \le \delta_k.
\]
Consequently, for any $k\in\NN$,
\[
\rho\enpar{\sum_{n=k}^\infty \kt_0^n \norm{f_n} } \le \ft \delta_k, 
\]
where $\Ft$ is a Fatou constant for $\rho$.

By \ref{it:PC}, $\sum_{n=1}^\infty \kt_0^n \norm{f_n} <\infty$ outside a null set $A$. Since
\[
\norm{\sum_{n=k}^m f_n} \le \sum_{n=k}^m \kt_0^n \norm{f_n}, \quad k,m\in\NN, \, k \le m,
\]
$\sum_{n=1}^\infty f_n(\omega)$ is a Cauchy series for all $\omega\in\Omega\setminus A$. Define $f(\omega)=\sum_{n=1}^\infty f_n(\omega)$ if $\omega\in\Omega\setminus A$ and $f(\omega)=0$ otherwise. We have
\[
\rho\enpar{\norm{f - \sum_{n=1}^{k-1} f_n}} \le \ft \delta_k, \quad k\in\NN.
\]
We infer that $\sum_{n=1}^\infty f_n$ converges to $f$.

\medskip
We close our approach to vector-valued spaces of measurable functions by proving that for every modular function norm $\rho$ and every quasi-Banach space $\XX$ 
\[L_\rho(\XX) \subseteq L_{0,f}(\mu,\XX)
\]
continuously. Given $E\in\Sigma(\mu)$, $t$, $\delta>0$, let $V_{E,\delta,t}$ be the neighbourhood at the origin for $L_0(\mu,\XX)$ defined as in \eqref{eq:NeiL0}. There is $\varepsilon>0$ such that $A\subseteq E$ and $\rho( u_E\chi_A)<\varepsilon$ implies $\mu(A)<\delta$. Let $f\in L(\mu,\XX)$ be such that $ u_E f/ t \in B_{\rho,\XX}(\varepsilon)$ and set
\[
A=E \cap \enbrace{\omega\in\Omega \colon \norm{f}>t}.
\]
We have $\rho(u_E \chi_A)\le \rho(u_E\norm{f}/t)<\varepsilon$. Consequently, $\mu(A)<\delta$. This proves that 
\[
\frac{t}{u_E} B_{\rho,\XX}(\varepsilon)\subseteq V_{E,\delta,t}.
\]
\medskip

Let us record some instances of modular function norms.  An \emph{Orlicz function} will be a non-decreasing left-continuous function
\[
F\colon[0,\infty)\to[0,\infty]
\]
such that $\lim_{t\to 0^+} F(t)=0$ and $F(\infty):=\lim_{\to\infty} F(t)>0$.
Given $t\in[0,\infty)$, a $\sigma$-finite measure space $(\Omega,\Sigma,\mu)$, and
\[
M\colon\Omega\times [0,\infty)\to [0,\infty],
\]
we denote by $\nu_M(\cdot,t)$ is the measure on $(\Omega,\Sigma,\mu)$ given by
\[
\nu_M(A,t)=\int_A M(\omega,t)\, d\mu(\omega), \quad A\in\Sigma.
\]
We say that $M$ is a Musielak-Orlicz function if
\begin{itemize}
\item $M(\omega,\cdot)$ is an Orlicz function for all $\omega\in\Omega$, and
\item for every $E\in\Sigma(\mu)$ there exists $u_E\in(0,\infty)$ such that
\[
\nu_M(E,u_E)<\infty
\]
\end{itemize}
and $M(\omega,u_E)>0$ for all $\omega\in E$ (cf.\@ \cite{Musielak1983}*{Chapter~7}).

If $f\colon\Omega\to[0,\infty)$ is a simple measurable function, then $M(\cdot,f(\cdot))$ is a measurable function. By left-continuity, $M(\cdot,f(\cdot))$ is measurable for all $f\in L^+(\mu)$. Hence, associated with the Musielak-Orlicz function $M$ there is a gauge 
\[
\rho_M\colon L^+(\mu)\to[0,\infty], \quad f\mapsto \int_\Omega M(\omega,f(\omega))\, d\mu(\omega).
\]

The monotonicity of the integral yields $\rho_M$ to fulfill \ref{FQN:M}. By the monotone convergence theorem, $\rho_M$ has the Fatou property and by the dominated convergence theorem, it is absolutely continuous. Since
\[
M(\omega,s+t)\le M(\omega, 2\max\{s,t\})=\max\enbrace{M(\omega,2s), M(\omega,2t)}
\]
for all $\omega\in\Omega$ and $s$, $t\in[0,\infty)$, $(1,2)$ is a convexity pair for $\rho_M$.

Let $\nu_{M,E}(\cdot,t)$ denote the restriction to $E\in\Sigma$ of $\nu_M(\cdot,t)$. Fix $E\in\Sigma(\mu)$. By assumption, $\rho_M(u_E\chi_E)<\infty$. Hence, $\nu_E:=\nu_{M,E}(\cdot,u_E)$ is a finite measure. Since $\mu|_E$ is absolutely continuous with respect to $\nu_{E}$, \ref{FQN:CM} holds.

Summing up, $\rho_M$ is a modular function norm. The complete metrizable topological vector space associated with $\rho_M$ and a quasi-Banach space $\XX$ is the vector-valued Musielak-Orlicz space $L_M(\XX)$ consisting of all $f\in L(\mu,\XX)$ such that
\[
\int_\Omega M(\omega,u\norm{f(\omega)})\, d\mu(\omega)<\infty
\]
for some $u>0$. The set $\{u B_M(\varepsilon) \colon u,\varepsilon>0\}$, where
\[
B_M(\varepsilon):=B_{\rho_M}(\varepsilon)=\enbrace{f\in L(\mu,\XX) \colon \int_\Omega M(\omega,\norm{f(\omega)})\, d\mu(\omega)<\varepsilon},
\]
is a neighbourhood basis at the origin for $L_M(\XX)$. We highlight two kinds of Musielak-Orllicz spaces built from functions satisfying special conditions.
\begin{itemize}
\item If $M$ is \emph{doubling}, that is there is $D\in(1,\infty)$ such that
\[
M(\omega,2s)\le D M(\omega,s), \quad (\omega,s)\in\Omega\times[0,\infty),
\]
then $\rho_M$ is pseudo-homogeneous at infinity.
Thus, 
\[
\enbrace{B_M(\varepsilon) \colon \varepsilon>0}
\]
is a a neighbourhood basis at the origin for $L_M(\XX)$, and
\[
L_M(\XX)=\enbrace{f\in L(\mu,\XX) \colon \int_\Omega M(\omega,\norm{f(\omega)})\, d\mu(\omega)<\infty}.
\]
\item If there are $c$, $d\in(0,1)$ such that
\begin{equation}\label{eq:PCO}
M(\omega,cs)\le d M(\omega,s), \quad (\omega,s)\in\Omega\times[0,\infty),
\end{equation}
then  $\rho_M$ is pseudo-homogeneous at the origin. This implies that $L_M(\XX)$ is a quasi-Banach space, and the map
\[
f \mapsto \inf\enbrace{ t>0 \colon \int_\Omega M\enpar{\omega,\frac{\norm{f(\omega)}}{t}} \, d\mu(\omega)<1}
\]
is a quasi-norm for $L_M(\XX)$. We point out that if there is $a\in(0,\infty)$ such that $M^{1/a}(\omega,\cdot)$ is convex for all $\omega\in\Omega$, then \eqref{eq:PCO} holds with any $0<c<1$ and $d=c^a$. Besides, $\rho_M$ is lattice $p$-convex with constant one for all $0<p\le a$. 
In turn, if  there is $b\in(0,\infty)$ such that $M^{1/b}(\omega,\cdot)$ is $q$-concave, then $p_M$ is lattice $q$-concave with constant one for all $q\ge b$.
\end{itemize}

\medskip
Let us prove that if $\mu$ is the counting measure on $\Omega$, and there is $u\in(0,\infty)$ such that 
\[
\inf_{\omega\in\Omega} M(\omega,u)>0,
\]

then the Musielak-Orlicz space $L_M(\mu,\XX)$ only depends on the behaviour of $M$ near the origin. To that end, we use that if $f\in B_{M}(\varepsilon)$, then
\[
M(\omega,\norm{f(\omega)})<\varepsilon
\]
for all $\omega\in\Omega$. Let $N$ be another Musielak-Orlicz function on the same measure space $(\Omega,\Sigma,\mu)$, and assume that there is $t_0\in(0,\infty)$ is such that $M(\omega,t)=N(\omega,t)$  for all $0<t\le t_0$. Pick $u_0\in[t_0,\infty)$ and $R_0\in(0,\infty)$ such that $M(\omega, u_0)\ge R_0$ for all $\omega\in \Omega$. If $0<\varepsilon \le R_0$, then 
\[
\frac{t_0}{u_0} B_M(\varepsilon)\subseteq B_N(\varepsilon).
\]
Indeed, if $f\in  B_M(\varepsilon) $, then $\norm{f(\omega)}\le u_0$ for all $\omega\in\Omega$. Therefore,
\[
\rho_N\enpar{ \frac{t_0}{u_0} \norm{f}} =\rho_M \enpar{ \frac{t_0}{u_0} \norm{f}} \le \rho_N\enpar{\norm{f}} <\varepsilon.
\]

Under certain assumptions on $M$, $L_M(\XX)$ only depends on the behaviour of $M$ near infinity. To be precise, if $M$ and $N$ are Musielak-Orlicz functions such that there is $t_0\in(0,\infty)$ for which $\nu_M(\cdot,t_0)$ is a finite measure, and $M(\omega,t)=N(\omega,t)$ for all $t\ge t_0$, then $L_M(\XX)=L_N(\XX)$. Indeed, given $0<u\le t_0$ and $f\in L(\mu,\XX)$,
\[
\rho_{N}\enpar{\norm{f}} \le \nu_N(\Omega,u) + \nu_N(\Omega_{f,u},t_0) +\rho_{M}\enpar{\norm{f}},
\]
where $\Omega_{f,u}$ is defined as in \eqref{eq:levelset}. A Chebyshev-type estimate gives
\[
\nu_M(\Omega_{f,u},u) \le \int_{\Omega_{f,u}} M\enpar{\omega, \norm{f(\omega)} } \, d\mu(\omega) \le \rho_M\enpar{\norm{f} }.
\]
Given $u$, $\varepsilon>0$, since $\nu_N(\cdot,t_0)$ is absolutely continuous with respect to $\nu_M(\cdot,u)$,  there is $\delta(\varepsilon,u)>0$ such that $\nu_N(A,t_0)<\varepsilon$ whenever $\nu_M(A,u)<\delta(\varepsilon,u)$. Summing up, if for a fixed $\varepsilon>0$ we pick $u\in(0,t_0]$ such that $\nu_N(\Omega,u)<\varepsilon/3$,
\[
B_M\enpar{\min\enbrace{\frac{\varepsilon}{3}, \delta\enpar{ \frac{\varepsilon}{3},u} } } \subseteq B_N(\varepsilon).
\]

The space $L_{0,f}(\mu,\XX)$ is a Musielak-Orlicz space. In fact, if $F$ is an Orlicz function with   
$F(\infty)=1$,
and
\[
\varphi\colon \Omega\to(0,\infty)
\]
is integrable and 
then $L_{0,f}(\mu,\XX)=L_M(\XX)$, where
\[
M(\omega,t)=\varphi(\omega) F(t), \quad   \omega\in \Omega, \, t\in[0,\infty),
\]
To prove this assertion, given $\varepsilon>0$ we choose $E\in\Sigma(\mu)$ with 
\[
\int_{\Omega\setminus E} \varphi \, d\mu<\frac{\varepsilon}{3},
\]
$t>0$ such that 
\[
F(t) \int_\Omega \varphi\, d\mu <\varepsilon/3,
\]
and $\delta>0$ such that $\int_B  \varphi \, d\mu<\varepsilon/3$ whenever $\mu(B)<\delta$. Since
\[
\rho_M(f)\le \int_{\Omega\setminus E} \varphi \, d\mu + \int_{\Omega_{f,t}\cap E}\,  \varphi \, d\mu +  F(t) \int_\Omega \varphi\, d\mu, \quad f\in L(\mu,\XX),
\]
$V_{E,\delta , t} \subseteq B_M(\varepsilon)$.

\medskip

Lebesgue spaces with variable exponents are Musielak-Orlicz spaces constructed from a measure space $(\Omega,\Sigma,\mu)$ and a measurable function $\pp\colon\Omega\to(0,\infty]$ via the Musielak-Orlicz function
\[
M(\omega,t)=t^{\pp(\omega)}, \quad \omega\in\Omega,\, t\in[0,\infty),
\]
with the convention that $t^\infty=\infty$ if $t>1$ and $t^\infty=0$ otherwise. Set $\Omega_{\pp}=\pp^{-1}((0,\infty))$. If there is $p>0$ such that $\pp(\omega)\geq p$ $\mu$-a.e.\@ $\omega\in\Omega$, then $\rho_M$ is lattice $p$-convex with constant one. Consequently, if $\XX$ is a $r$-Banach space, $0<r\le 1$, then $L_M(\XX)$ equipped with the quasi-norm
\[
f\mapsto \inf\enbrace{t>0 \colon \supess_{\omega\in\Omega\setminus\Omega_p}\norm{f(\omega)}<t, \, \int_{\Omega_p} \enpar{\frac{\norm{f(\omega)}}{t}}^{\pp(\omega)}\, d\mu(\omega)<1}.
\]
is a $\min\{r,p\}$-Banach space. In turn, if  there is $q<\infty$ such that $\pp(\omega)\leq q$ $\mu$-a.e.\@ $\omega\in\Omega$, then $\rho_M$ is lattice $q$-concave with constant one.
These spaces are being studied in depth when $\XX$ is a Banach space and $\pp(\omega)\ge 1$ for all $\omega\in\Omega$ so that the space 
\[L_{\pp}(\XX):=L_M(\XX)
\]
is a Banach space \cite{DHH2R017}. If $\mu$ is the counting measure, Lebesgue spaces of variable exponents are Bourgin spaces (see \cite{Bourgin1943}). If, besides, $\pp(\omega)\ge 1$ for all $\omega\in\Omega$, Lebesgue spaces of variable exponents are Nakano spaces (see \cite{Nakano1950}).

\medskip

We can regard Orlicz spaces as particular cases of Musielak-Orlicz spaces. Given a measure space $(\Omega,\Sigma,\mu)$, an Orlicz function $F$, and a quasi-Banach space $\XX$, the $\XX$-valued Orlicz space $L_F(\mu,\XX)$ is the Musielak-Orlicz space constructed from the Musielak-Orlicz function
\[
M_F\colon\Omega\times[0,\infty) \to [0,\infty), \quad M_F(\omega,t)=F(t).
\]
The associated gauge $\rho_{\mu,F}=\rho_{M_F}$ is given by
\[
\rho_{\mu,F}(f) =\int_\Omega F(f(\omega)) \, d\mu(\omega), \quad f\in L^+(\mu),
\]
the Orlicz space $L_F(\mu,\XX)$ is 
\[
\enbrace{f\in L(\mu,\XX) \colon \int_\Omega F\enpar{u\norm{f(\omega)}} \, d\mu(\omega)<\infty \mbox{ for some }u>0},
\]
and the balls of  $L_F(\mu,\XX)$ are the sets
\[
B_{\mu,F}(\varepsilon)=B_{M_F}(\varepsilon) =\enbrace{f\in L(\mu,\XX) \colon \int_\Omega \norm{f(\omega)} \, d\mu(\omega)<\varepsilon}, \quad \varepsilon>0.
\]
If $F$ is doubling, then $(B_{\mu,F}(\varepsilon))_{\varepsilon>0}$ is a neighbourhood basis of the origin, and 
\[
L_F(\mu,\XX)=\enbrace{f\in L(\mu,\XX) \colon \int_\Omega F\enpar{\norm{f(\omega)}} \, d\mu(\omega)<\infty}.
\]
If there are $c$, $d\in(0,1)$ such that $F(ct)\le d F(t)$ for all $t\in[0,\infty)$, then $L_F(\mu,\XX)$ is locally bounded, and
\[
f \mapsto \inf\enbrace{ t>0 \colon \int_\Omega F\enpar{\frac{\norm{f(\omega)}}{t}} \, d\mu(\omega)<1}
\]
is a quasi-norm for  $L_F(\mu,\XX)$. If $\mu$ is the counting measure, then $L_F(\mu,\XX)$ only depends on the behaviour of $F$ near the origin. In turn, if $\mu$ is finite, then $L_F(\mu,\XX)$ only depends on the behaviour of $F$ near infinity. If $\mu$ is finite and $F$ is bounded then $L_F(\mu,\XX)=L_{0}(\mu,\XX)$.

\end{example}
\bibliography{references}
\bibliographystyle{plain}
\end{document}